\documentclass[reqno,final,12pt]{amsart}
\usepackage{amssymb,amsthm,amsmath,color,fullpage,url}
\usepackage{todonotes}
\usepackage{comment}
\usepackage[capitalise, compress, nameinlink, noabbrev]{cleveref}
\crefname{lem}{Lemma}{Lemmas}
\crefname{thm}{Theorem}{Theorems}
\crefname{obs}{Observation}{Observations}
\crefname{prop}{Proposition}{Propositions}
\crefname{conj}{Conjecture}{Conjectures}
\usepackage{tikz}
\usetikzlibrary{patterns}
\usetikzlibrary{decorations.pathreplacing}

\newtheorem{thm}{Theorem}
\newtheorem{lem}[thm]{Lemma}
\newtheorem{obs}[thm]{Observation}

\newtheorem{conj}[thm]{Conjecture}

\newtheorem{claim}{Claim}[section]

\newcommand{\mc}[1]{\mathcal{#1}}

\begin{document}
	
	\title{Every  graph with no $K_7^{\vee}$-minor is $6$-colorable}
	
	\author{
		Sergey Norin	
		and Agn\`es Totschnig}
	\address{Department of Mathematics and Statistics, McGill University, Montr\'{e}al, QC, Canada.}
	\thanks{SN was partially supported by NSERC Discovery Grant. AT was supported by a McCall MacBain Scholarship and partially by FRQNT Grant 332481.
	}
	
\begin{abstract} 
	Let  $K_7^{\vee}$ denote the graph obtained from the complete graph on seven vertices by deleting two edges with a common end. Motivated by Hadwiger's conjecture, we prove that every  graph with no $K_7^{\vee}$-minor is $6$-colorable.
\end{abstract}
	\maketitle
	\section{Introduction}
	
	In $1943$ Hadwiger made the following famous conjecture.
	
	\begin{conj}[Hadwiger's conjecture~\cite{Hadwiger1943}]\label{Hadwiger} For every integer $t \geq 1$, every graph with no $K_{t}$ minor is $(t-1)$-colorable. 
	\end{conj}
	
	We refer the reader to a comprehensive survey by Seymour~\cite{Seymour2016} of progress towards the conjecture, and only give a brief overview of the results closest to our investigation. 
	
	Hadwiger~\cite{Hadwiger1943} and Dirac~\cite{Dirac1952} independently showed that Conjecture~\ref{Hadwiger} holds for $t \leq 4$. Wagner~\cite{Wagner1937} proved that for $t=5$  Hadwiger's  conjecture is equivalent to the Four Color Theorem, which was subsequently proved by Appel and Haken~\cite{AH1977,AHK1977}. 
	Robertson, Seymour and  Thomas~\cite{RST1993} proved Hadwiger's conjecture for $t=6$, also by reducing it to the Four Color Theorem. 
	
	In this paper we are concerned with the next open case: $t=7$. 
	Settling the conjecture in this case appears to be extremely challenging. Albar and Gon\c{c}alves~\cite{AG2018} proved that every graph with no $K_7$ minor is $8$-colorable, but even the question whether every such graph is $7$-colorable is open. (Rolek, Song and  Thomas discuss properties of minor minimal non-$7$-colorable graphs with no $K_7$-minor in \cite{RST2023}.)
	
	In another direction, several authors explored minors that are guaranteed to be present in every non-$6$-colorable graph. The following two theorems of Jakobsen~\cite{Jakobsen1971} and of Kawarabayashi and Toft~\cite{KT2005}, in particular, inspired our results. 
	
	Stating Jakobsen's result requires the following additional notation. For $t \geq 4$, let $K^{\vee}_t$ denote the graph obtained from $K_t$ by deleting two edges with a common end, and let $K^{=}_t$  denote the graph obtained from $K_t$ by deleting a two-edge matching. Thus every graph obtained from $K_t$ by deleting two edges is isomorphic to $K^{\vee}_t$ or $K^{=}_t$.
	
		\begin{thm}[Jakobsen{~\cite{Jakobsen1971}}]\label{t:Jakobsen}
		Every  graph with no $K_7^{\vee}$-minor and no $K^{=}_7$-minor is $6$-colorable.
	\end{thm}
	
	\begin{thm}[Kawarabayashi and Toft{~\cite{KT2005}}]\label{t:KTmain}
		Every  graph with no $K_7$-minor and no $K_{4,4}$-minor is $6$-colorable.
	\end{thm}
	
	Our main theorem strengthens Jakobsen's result, showing that excluding $K^{=}_7$ is not necessary.
	 
	\begin{thm}\label{t:main}
		Every  graph with no $K_7^{\vee}$-minor is $6$-colorable.
	\end{thm}

	The proof of \cref{t:main} uses the same overall strategy as the proofs of Theorems~\ref{t:Jakobsen} and~\ref{t:KTmain}, which is also deployed to some extent in the proofs of the Four Color Theorem and Hadwiger's conjecture for $t=6$. However,  we have to substantially extend the toolkit used in the proofs of these earlier results. We hope that our techniques could be further applied to closer approach Hadwiger's conjecture for small values of $t$. 
	
	First, we restrict our attention to \emph{$7$-contraction-critical} graphs, i.e. graphs $G$ such that $\chi(G)=7$ and  $\chi(G') < 7$ for every proper minor $G'$ of $G$.  Second,  we determine the maximum number of edges in $G$, as a function of the number of vertices. 
	In the proof of \cref{t:KTmain}, the corresponding part of the argument uses the following result of J{\o}rgensen.
	
	\begin{thm}[J{\o}rgensen~\cite{Jorgensen1994}]\label{t:Jorgensen}
		Every $4$-connected graph $G$ with $|E(G)| \geq 4|V(G)| - 7$ has a $K_{4,4}$-minor, unless $G \simeq K_7$.
	\end{thm}
	
	The main technical part in the proof of \cref{t:main} is establishing the following analogue of \cref{t:Jorgensen} in our setting. 		
	
	\begin{thm}\label{t:extremal}
		Every $4$-connected graph $G$ with $|E(G)| \geq 4|V(G)| - 8$ has a $K_7^{\vee}$-minor, unless $G \simeq K_{2,2,2,2}$.
	\end{thm}
	
In \cref{s:rooted}, as a step towards the proof of \cref{t:extremal}, we obtain extremal results guaranteeing the existence of certain ``rooted'' minors, using and refining theorems of Robertson, Seymour and Thomas~\cite{RST1993} and J{\o}rgensen~\cite{Jorgensen1994}.  In \cref{s:extremal} we prove \cref{t:extremal}.

In \cref{s:main} we implement the remainder of the strategy and complete the proof of \cref{t:main}. The brief outline of the argument is as follows.  By \cref{t:extremal}, we can conclude that $G$ contains many degree seven vertices. The neighborhood of every degree seven vertex contains a clique of size five, except for one case, which we dismiss using a recent theorem of Kriesell and Mohr~\cite{KM2019}. The resulting cliques can than be connected to obtain the required minor. This last part closely follows the argument in~\cite{KT2005}, which is in turn inspired by~\cite{RST1993}.  We conclude with several open questions in~\cref{s:remarks}.

\subsection*{Notation} We use standard graph theoretical notation. We denote the subgraph of a graph $G$ induced by a set $X \subseteq V(G)$ by $G[X]$. 

A \emph{separation} of a graph $G$ is a pair $(A,B)$  such that $A \cup B = V(G)$ and $G[A] \cup G[B]= G$, i.e. no edge of $G$ has one end in $A-B$ and the other in $B-A$. The \emph{order} of a separation $(A,B)$ is $|A \cap B|$. A separation $(A,B)$ is \emph{non-trivial} if $A,B \neq V(G)$.  A \emph{$k$-separation} is a separation
of order $k$.

We write $G \simeq H$ if the graphs $G$ and $H$ are isomorphic. We denote by $N(v)$ the set of all vertices adjacent to the vertex $v$ in a graph $G$ (where $G$ is understood from context, whenever we use this notation). Let $N[v]= N(v) \cup \{v\}$.

We denote the graph obtained from the complete graph $K_t$ by deleting a single edge by $K^-_t$.
	
\section{Rooted models}\label{s:rooted}

Many of the results we use in this paper are best stated in terms of rooted models, which we now define. 

Let $k$ be a non-negative integer. A \emph{$k$-rooted graph} is a $(k+1)$-tuple $(G,v_1,\ldots,v_k)$ consisting of a graph $G$ and an ordered sequence of pairwise distinct vertices of $G$. Let $(H,u_1,\ldots,u_k)$ and $(G,v_1,\ldots,v_k)$ be two $k$-rooted graphs. 
A \emph{$(v_1,\ldots,v_k)$-rooted}  \emph{$(H,u_1,\ldots,u_k)$-model  in $G$} is a function $\alpha$ assigning to each $u \in V(H)$ a set $\alpha(u) \subseteq V(G)$  such that 
\begin{itemize}
\item $\alpha(u) \cap \alpha(u') = \emptyset$ for all $u,u' \in V(H)$,  $u \neq u'$,
	\item the subgraph $G[\alpha(u)]$, which we refer to as \emph{a bag of this model}, is non-null and connected for every $u \in V(H)$, 
	\item for every $uu' \in E(H)$ there exists an edge of $G$ with one end in $\alpha(u)$ and another in $\alpha(u')$, and
	\item $v_i \in \alpha(u_i)$ for every $i \in [k]$.
\end{itemize}
Note that contracting all the bags of the model to single vertices we obtain a minor of $G$ with a subgraph isomorphic to $H$ induced by the vertices resulting from contraction, and, in particular, an edge joining vertices corresponding to $v_i$ and $v_j$, whenever $u_i$ and $u_j$  are adjacent in $H$.

If $H$ and $G$ are graphs $X \subseteq V(H), Z \subseteq V(G)$ with $|X|=|Z|=k$ then a \emph{$Z$-rooted $(H,X)$-model in $G$} is a {$(v_1,\ldots,v_k)$-rooted $(H,u_1,\ldots,u_k)$-model  in $G$ for some ordering $(v_1,\ldots,v_k)$ of $Z$ and  $(u_1,\ldots,u_k)$ of $X$.

We will consider rooted models in three particular cases:
\begin{itemize}
	\item $|V(H)|=k$, in which case we refer to an $(H,V(H))$-model as a \emph{an $H$-model}, for brevity;
	\item $V(H)$ admits a partition $(A,B)$ such that $|A|=k,$ $A$ is independent, every vertex of  $A$ is adjacent to every vertex of $B$, and either $B$ is independent or $B$ is a clique. In the first case we refer to a rooted $(H,A)$-model as a \emph{rooted $K_{k,|B|}$-model} and in the second as a \emph{rooted $K^*_{k,|B|}$-model};
	\item $H$ is  a cycle of length $k$ with vertices $u_1,\ldots,u_k$ in order. Then we refer to a rooted $(H,u_1,\ldots,u_k)$-model as a  \emph{rooted $C_k$-model}.
\end{itemize}

Let $G$ be a graph, let $Z \subseteq V(G)$, and let $k \geq 0$ be an integer. We say that a separation $(A,B)$ of $G$ is \emph{a separation of $(G,Z)$}  if $Z \subseteq A$. We say that a pair $(G,Z)$ is \emph{internally $k$-connected} if there exists no separation $(A,B)$ of $(G,Z)$ of order less than $k$ such that $B-A \neq \emptyset$.  The following easy observation connects the properties of internally $k$-connected pairs and $k$-connected graphs.

\begin{obs}\label{o:intConnect} Let $G$ be a graph and let  $k \geq 0$ be an integer.	
	 \begin{enumerate}  
		\item Suppose that $|V(G)| \geq k+1$ and let $Z \subseteq V(G)$. Then $(G,Z)$ is internally  $k$-connected if and only if the graph obtained from $G$ by adding edges between all pairs of non-adjacent vertices of $Z$ is $k$-connected.
		\item If $G$ is $k$-connected and $(A,B)$ is a separation of $G$ of order at least $k$ then $(G[B],A~\cap~B)$ is internally $k$-connected. 
	\end{enumerate}	
\end{obs}

\begin{proof} \qquad 
	
	\begin{enumerate} \item
	Let $G^*$ be the graph obtained from $G$ 	by adding edges between all pairs of non-adjacent vertices of $Z$. Then for every separation $(A,B)$ of $G^*$ we have $Z \subseteq A$ or $Z \subseteq B$. Thus, either $(A,B)$ or $(B,A)$ is a separation of $(G,Z)$. If $(G,Z)$ is internally $k$-connected, this implies that every non-trivial separation of $G^*$ has order at least $k$ and so $G^*$ is $k$-connected. 
	
	Conversely, any separation $(A,B)$ of $(G,Z)$ is a separation of $G^*$. 
    In particular, every separation $(A,B)$ of $(G,Z)$ with $B-A \neq \emptyset$ is either a non-trivial separation of $G^*$, or is of the form $(Z',V(G))$ for some $Z' \supseteq Z$. In either case, if  $G^*$ is $k$-connected, such a separation has order at least $k$, implying that $(G,Z)$ is internally $k$-connected.
	
	\item Let $(A',B')$ be a separation of  $(G[B],A \cap B)$  such that $B' - A' \neq \emptyset$. Then $(A  \cup A', B')$ is a separation of $G$ of the same order as $(A',B')$. If $(A  \cup A', B')$  is non-trivial then it follows that the order of  $(A',B')$  is at least $k$ by $k$-connectivity of $G$. Otherwise $B' = V(G)$ and so $A \cap B \subseteq A' \cap B'$, again implying that the order of  $(A',B')$  is at least $k$.
		\end{enumerate}	
\end{proof}

Our first result on rooted models determines the minimum number of edges sufficient to guarantee a $Z$-rooted $K_4$-model for an internally $4$-connected pair $(G,Z)$. We derive it from a structural result of Robertson, Seymour and Thomas. Stating their result requires the following definition.

 A \emph{trisection} of a graph $G$ is a triple $(A,B,C)$ of subsets of $V(G)$ such that
$A \cap B = A \cap C = B \cap C$ and $G[A] \cup G[B] \cup G[C]=G$. The \emph{order of a trisection $(A,B,C)$} is $|A \cap B \cap C|$.

\begin{thm}[Robertson, Seymour and Thomas {\cite[(2.6)]{RST1993}}]
	\label{t:RST26}
	Let $Z \subseteq V(G)$ with $|Z|=4$. Then either
	\begin{itemize}
		\item[(i)] $G$ has a $Z$-rooted $K_4$-model, or
		\item[(ii)] there is a trisection $(A_1,A_2,B)$ of order two such that $|Z \cap (A_i \setminus B)| = 1$ for $i = 1, 2$, or
		\item[(iii)] there is a separation $(A,B)$ of $(G,Z)$ of order at most three such that $|B \setminus A| \geq 2$ and $|Z \cap B| \leq 2$, or
		\item[(iv)] $G$ can be drawn in a plane so that every vertex in $Z$ is incident with the infinite region.
	\end{itemize}
\end{thm}

We use the following consequence of \cref{t:RST26}. 
 	 
 \begin{lem}\label{l:K4}
		Let $G$ be a graph, and let $Z \subseteq V(G)$ with $|Z|=4$ be such that $(G,Z)$ is internally $4$-connected. If $G$ has no $Z$-rooted $K_4$-model then $$|E(G)| \leq 3|V(G)|-7.$$ 
	\end{lem}
	
\begin{proof} We prove the lemma by induction on $|V(G)|$. The base case ($|V(G)| = 4$) holds, as in this case $|E(G)| \leq 3|V(G)|-7 =5$, or $Z$ induces $K_4$ in $G$.
	
	 For the induction step we have $|V(G)| \geq 5$. We apply \cref{t:RST26} to $G$ and $Z$. If outcome (i) holds then $G$ has a $Z$-rooted $K_4$-model. 
	 
	 If (iv) holds then we can add an edge to a planar drawing of $G$ to triangulate the infinite region, and using the bound on the maximum number of edges in the planar graph we conclude  $|E(G)| + 1 \leq 3|V(G)|-6$, i.e.   $|E(G)| \leq 3|V(G)|-7$. 
	 
	 As $(G,Z)$ is internally $4$-connected, (iii) cannot hold.
	
 	Finally, suppose that (ii) holds and  there is a trisection $(A_1,A_2,B)$ of $G$ of order two such that $|Z \cap (A_i \setminus B)| = 1$ for $i = 1, 2$. Let $Z'= A_1 \cap A_2 \cap B$.
	If $A_i \setminus Z' \setminus Z \neq \emptyset$ for some $i$ then the separation $(B \cup A_{3-i} \cup Z, A_i)$  contradicts internal $4$-connectivity of $(G,Z)$. Thus $A_i \setminus Z'$ consists of a single vertex of $Z$ for $i =1,2$. If $|B| \leq 4$, then $|E(G[B])| \leq \binom{|B|}{2} \leq 3|B|-6$. Therefore \begin{equation}\label{e:GB}|E(G)| \leq |E(G[B])| + 4 \leq 3|B|-2 = 3(|V(G)|-2)-2 \leq 3|B|-8.\end{equation}
	Thus we assume $|B| \geq 5$. In particular, by internal $4$-connectivity of $(G,Z)$, the separation $(A_1 \cup A_2 \cup Z, B )$ has order at least four, i.e. $$ |Z' \cup (Z \cap B)| = |(A_1 \cup A_2 \cup Z) \cap B| \geq 4.$$
	As $|Z'|=|Z \cap B| = 2$, it follows that $Z' \cap Z \cap B = \emptyset$. 
	
Let $Z''= Z-B = (A_1 \cup A_2) - B$. As $(G,Z)$ is internally $4$-connected and $V(G)-Z \neq \emptyset$ every vertex in $Z'' \subseteq Z$ has a neighbor in $V(G)-Z$ and  these neighbors must lie in $(A_1 \cup A_2) - Z = Z'$. Similarly every vertex of $Z'$ has a neighbor in $Z''$.  Thus we may assume that $Z''=\{a_1,a_2\},$ $Z'=\{b_1,b_2\}$ and $a_ib_i \in E(G)$ for $i=1,2$. If $G[B]$ contains a 
$(Z' \cup (B \cap Z))$-rooted $K_4$-model, this model could be extended to a $Z$-rooted $K_4$-model in $G$ using these two edges.

Thus  we assume that $G[B]$ contains no such model and by the induction hypothesis we have $|E(G[B])| \leq 3|B|-7$. Analogously to \eqref{e:GB} we conclude $|E(G)| \leq 3|V(G)|-9 < 3|V(G)|-7$, a contradiction, finishing the proof of the lemma.  
\end{proof}	

The next tool we will use in \cref{s:extremal} is a reformulation of a lemma of J{\o}rgensen.

\begin{lem}[J{\o}rgensen {\cite[Lemma 16 (2)]{Jorgensen1994}}]\label{c:RST64} Let $G$ be a graph with $|V(G)| \geq 6$, and let $Z \subseteq V(G)$ with  $|Z|=4$ be such that $(G,Z)$ is internally $4$-connected. Then $G$ has a $Z$-rooted $K^-_4$-model.
\end{lem}

Our third result on rooted minors is an extension of another result of  J{\o}rgensen on the extremal function of rooted $K_{4,2}$-models. (See~\cite{Wollan2008} for generalization of this result to $K_{t,2}$-models.)

\begin{thm}[Jorgensen~{\cite[Lemma 17]{Jorgensen1994}}]\label{l:K42}
	Let $G$ be a graph, and let $Z \subseteq V(G)$ with $|Z|=4$ be such that $(G,Z)$ is internally $4$-connected. If $G$ has no $Z$-rooted $K_{4,2}$-model then $$|E(G)| \leq 4|V(G)|-10.$$ 
\end{thm}

We will improve on \cref{l:K42} by replacing the  $K_{4,2}$-model by a $K^*_{4,2}$-model in the conclusion.

 \begin{lem}\label{l:K42star}
	Let $G$ be a graph, and let $Z \subseteq V(G)$ with $|Z|=4$ be such that $(G,Z)$ is internally $4$-connected. If $G$ has no $Z$-rooted $K^*_{4,2}$-model then $$|E(G)| \leq 4|V(G)|-10.$$ 
\end{lem}
	
\begin{proof} 
	Suppose for a contradiction that $|E(G)| \geq 4|V(G)|-9,$  $G$ has no rooted  $Z$-rooted $K^*_{4,2}$-model and choose such a pair $(G,Z)$ satisfying the conditions of the lemma with $|V(G)|$ minimum.
	
	By \cref{l:K42} the graph $G$ has a $Z$-rooted $K_{4,2}$-model, i.e. there exist  six pairwise vertex disjoint non-null connected subgraphs $H_1,H_2,H_3,H_4,J_1$ and $J_2$ of $G$ such that for all $1 \leq i \leq 4, 1 \leq j \leq 2$ \begin{itemize}
		\item $|V(H_i) \cap Z|=1$, and
		\item there exists an edge of $G$ with one end in $V(H_i)$ and another in $V(J_j)$.
	\end{itemize} 
Choose such subgraphs with $H_i$ minimal and $J_j$ maximal for all $i$ and $j$.  Then each $H_i$ is a tree. 
	
	If $|V(H_i)|=1$, let $v_i$ be its vertex. Otherwise, let $v_i \in V(G) - Z$ be a leaf of $H_i$. Suppose that $v_i$ is not the only vertex of $H_i$ with a neighbor in $J_1$. Then $v_i \not \in Z$. If $v_i$ has no neighbor in $J_{2}$ then $H_i \setminus v_i$ contradicts minimality of $H_i$. Otherwise, we can  add $v_i$ to $J_{2}$ and remove it from $H_i$ once again contradicting the choice of $H_i$. 	Thus $v_i$ is  the only vertex of $H_i$ with a neighbor in $J_1$. In particular, this implies that $H_i$ is a (possibly trivial) path from $Z$ to $v_i$. 
	
	Further, by maximality of $J_1$ vertices of $J_1$ have no neighbors in $V(G) - \cup_{i=1}^4 V(H_i) - V(J_1) -V(J_2)$. If there is an edge of $G$ with one end in $V(J_1)$ and another in $V(J_2)$ then $G$ has a   $Z$-rooted $K^*_{4,2}$-model, so  no such edge exists. 
	
	It follows that all the neighbors of vertices of  $J_1$ in $V(G) -V(J_1)$ are in the set $Z'=\{v_1,v_2,v_3,v_4\}$. Thus $(V(G) - V(J_1), V(J_1) \cup Z')$ is a separation of $G$ of order four. 	 If $G[V(J_1) \cup Z']$ has a $Z'$- rooted $K^*_{4,2}$-model then it can be extended to a  $Z$-rooted $K^*_{4,2}$-model in $G$ using the paths $H_1,\ldots, H_4$. Thus we assume that no such model exists, and by the induction hypothesis we have $$|E(G[V(J_1) \cup Z'])|  \leq 4(|V(J_1)|+4)-10 =  4|V(J_1)|+ 6,$$ implying \begin{equation}\label{e:J1} |E(G[V(J_1) \cup Z'])|- |E(G[Z'])| \leq 4|V(J_1)|.\end{equation}
	By symmetry, the analogous statement holds for $J_2$.
	
	Suppose that $|V(J_1)| \neq 1$. Then by \cref{c:RST64} the graph  $G[V(J_1) \cup Z']$ has a $Z'$-rooted $K^-_4$-model, and contracting $V(J_1)$ onto $Z'$ we obtain a minor of $G$ inducing at least five edges on $Z'$. Further contracting $V(J_2)$ onto $Z'$ we obtain a minor $G'$ of $G$ with $V(G')=V(G)-V(J_1)-V(J_2)$ such that $G'[Z']$ is complete. Moreover,  \begin{align*}|E(G')| &\geq |E(G)| - \sum_{i=1}^2(|E(G[V(J_i) \cup Z'])|- |E(G[Z'])|) \\ &\stackrel{\eqref{e:J1}}{\geq} |E(G)| - 4|V(J_1)|- 4|V(J_2)| \geq 4|V(G')|-9, \end{align*}
	$(G',Z)$ is internally $4$-connected, and $G'$ has no $Z$-rooted $K^*_{4,2}$-model, contradicting the choice of $G$. 
	
	It remains to consider the case  $|V(J_1)| = 1$. By symmetry we also assume $|V(J_2)|=1$, and let $V(J_i)=\{u_i\}$ for $i=1,2$. Let $G'=G/u_1v_1$. We have $|E(G')| \geq |E(G)|-4 \geq 4|V(G')|-9,$ and $G'$ has no $Z$-rooted $K^*_{4,2}$-model. It follows that $(G',Z)$ is not internally $4$-connected, i.e. there exists a separation $(A',B')$ of $G'$ of order at most three with $Z \subseteq A'$, $B'-A' \neq \emptyset$. Let $(A,B)$ be the corresponding separation of $G$. Then $u_1,v_1 \in A \cap B$,  as otherwise  $(A,B)$ contradicts the internal $4$-connectivity of $(G,Z)$, moreover $u_1$ must have neighbors in both  $A-B$ and $B-A$, as otherwise we could remove it from $A$ or $B$, again obtaining a contradiction. As $u_2$ has the same neighbors as $u_1$ it follows that $u_2 \in A \cap B$. 
	
	Suppose now that $u_1$ has a unique neighbor in $A-B$, say $v_2$. Then $(A-\{u_1,u_2\},B \cup \{v_2\})$ is a separation of $(G,Z)$ of order at most three, again contradicting the internal $4$-connectivity. Thus $u_1$ has at least two neighbors in $A-B$ and so at most one neighbor in $B-A$. Assume without loss of generality that $v_3$ is this neighbor. If $B-A \neq \{v_3\}$ then $(A \cup \{v_3\},B -\{u_1,u_2\})$ contradicts the internal $4$-connectivity. 
	
	Finally, we may assume  $B-A = \{v_3\}$. It follows that $v_3$ is adjacent to all the vertices of $A\cap B$. In particular, $v_1v_3 \in E(G)$. We now consider a minor $G'' = G /u_1v_2 / u_2v_4$ of $G$. The subgraph $G''[Z']$ is complete, which implies that $(G'',Z)$ is  internally $4$-connected. We have $|E(G'')| \geq |E(G)|-8 \geq 4|V(G'')|-9.$ Thus $G''$ contradicts the choice of $G$. This contradiction finishes the proof.   
	\end{proof}

Finally, we will use a result on existence of a pair of ``crossing'' disjoint paths. 

\begin{thm}[Robertson, Seymour and Thomas{\cite[(2.4)]{RST1993}}]
\label{l:RST24}
Let $v_1, \dots, v_k$ be distinct vertices of a graph $G$. Then either
\begin{itemize}
	\item[(i)] there are vertex disjoint paths in $G$ with ends $p_1 p_2$ and $q_l q_2$ respectively, so that $p_1, q_1, p_2, q_2$ occur in the sequence $v_1, \dots, v_k$ in order, or
	\item[(ii)] there is a separation $(A,B)$ of $G$ of oreder at most three with $v_1, \dots, v_k \in A$ and $|B \setminus A| \geq 2$, or
	\item[(iii)] $G$ can be drawn in a disc with $v_1, \dots, v_k$ on the boundary in order.
\end{itemize}
\end{thm}

Our final tool involving rooted minors, which we use in \cref{s:main}, is a result of Kriesell and Mohr that allows one convert Kempe chains into minors. Let $c: V(G) \to S$ be a (proper) coloring of a graph $G$.  A \emph {Kempe chain of a coloring $c$} is a connected component of $c^{-1}(s_1) \cup c^{-1}(s_2)$ for a pair of distinct $s_1,s_2 \in S$.   

\begin{thm}[Kriesell and Mohr~{\cite[Lemma 2]{KM2019}}]	\label{t:KM2019} Let  $c: V(G) \to S$ be a coloring of a graph $G$ and let $v_1,v_2,\ldots,v_k \in V(G)$ be such that $c(v_i) \neq c(v_j)$ for $i \neq j$. Suppose that for each $1 \leq  i \leq k$ there exists a Kempe chain of $c$ containing $v_{i}$ and $v_{i+1}$. (We define $v_{k+1}=v_1$.)  Then $G$ contains a $(v_1,\ldots,v_k)$-rooted $C_k$-model.
\end{thm}

\section {Proof of \cref{t:extremal}.}\label{s:extremal}
	
In this section we prove \cref{t:extremal}. We assume for a contradiction that there exists a graph $G$ such that 
\begin{itemize}
	\item $G$ is $4$-connected,
	\item $|E(G)| \geq  4|V(G)| - 8,$
	\item $G$ has no $K_7^{\vee}$-minor,
	\item $G$ is not isomorphic to $K_{2,2,2,2}$.
\end{itemize}
For brevity let us call every graph satisfying this property an \emph{enemy}. 

Let $G$ be a minor-minimal enemy. We proceed to establish properties of $G$ and eventually a contradiction via a series of claims.

\begin{claim}\label{c:vertexnumber} $|V(G)| \geq 9$.\end{claim}	
\begin{proof}
As $G$ is $4$-connected, $|V(G)| \geq 5$. As  $\binom{|V(G)|}{2} \geq |E(G)| \geq  4|V(G)| - 8,$ we have $|V(G)| \geq 7$. If $|V(G)|=7$ then $G$ has a $K^-_7$ subgraph and thus a $K_7^{\vee}$ subgraph, a contradiction.

Finally, if $|V(G)|=8$ then  $G$ has at most four non-edges. If these non-edges form a matching then either  $G \simeq K_{2,2,2,2}$ or contracting an edge of $G$ we obtain $K^-_7$ or $K_7$. Thus we assume that $\deg(v) \leq 5$ for some $v \in V(G)$.  There exists  $u \in V(G)$ such that $u$ is adjacent to $v$ and all non-neighbors of $v$. The graph $G / uv$ then has at most two non-edges and contains a   $K_7^{\vee}$ subgraph, unless the non-edges of $G$ are of the form $\{v,v_1\},\{v,v_2\},\{u_1,u_2\},\{u_3,u_4\}$ where $v,v_1,v_2,u_1,\ldots,u_4$ are pairwise distinct. In this last case $G / u_1u_3 \simeq K_7^{\vee}$, a contradiction, implying the claim. 
\end{proof}

\begin{claim}\label{c:no4sep2K4} There exists no non-trivial $4$-separation $(A,B)$ of $G$ such that $G[A]$ and $G[B]$ both have an $(A \cap B)$-rooted $K_4$-minor.
\end{claim}	

	\begin{proof} Suppose that such a separation $(A,B)$ exists. Let $k=|E(G[A \cap B])|$. Let $G_A$ and $G_B$ be the graphs obtained from $G[A]$ and $G[B]$, respectively, by adding $6-k$ missing edges of $G[A \cap B]$. Thus $A \cap B$ is a clique in $G_A$ and $G_B$.

By Observation \ref{o:intConnect} both $G_A$ and $G_B$ are $4$-connected. 	
		
As $G_A$ and $G_B$ are proper minors of $G$, neither of them is an enemy by the choice of $G$. 	Thus for $X \in \{A,B\}$ we have $|E(G_X)| \leq 4|X| - 9+\chi_X$ , where  $\chi_X =1$ if $G_X \simeq K_{2,2,2,2}$  and $\chi_X=0$, otherwise.
It follows that   \begin{align*}
	|E(G)| &= |E(G[A])| + |E(G[B])| -k = |E(G_A)|+|E(G_B)| -2(6-k)-k \\
	&\leq 4|A| - 9  + 4|B| - 9 - 12 + k + \chi_A +\chi_B 
	= 4(|V(G)| + 4) - 30 + k  + \chi_A +\chi_B \\
	&\leq 4|V(G)| - 14 + k  + \chi_A +\chi_B,
\end{align*} 
As $E(G) \geq  4|V(G)| - 8,$ it follows that $k+\chi_A +\chi_B \geq 6$ and so \begin{itemize} \item either $k = 6$, or \item $k=5$ and either  $G_A \simeq K_{2,2,2,2}$ or $G_B \simeq K_{2,2,2,2}$
 or \item $k=4$ and  $G_A \simeq K_{2,2,2,2}$ and $G_B \simeq K_{2,2,2,2}$.
 \end{itemize}
 \noindent {\bf Case 1.1: $k=6$.}
 Note that in this case neither $G[A]$ nor $G[B]$ have an $(A \cap B)$-rooted $K^*_{4,2}$-minor.
 Indeed, suppose without loss of generality that $G[A]$ has such a minor. Thus we can contract the edges of $G[A]$ to obtain a $K_6$ with four of the vertices corresponding to $A \cap B$.  Contracting any connected component of $B-A$ to a single vertex and deleting the remaining vertices of $B-A$ we obtain $K_7^{\vee}$ as a minor of $G$, a contradiction.
 
  It follows from \cref{l:K42star} that $|E(G[X])| \leq 4|X|- 10$ for $X \in \{A,B\}$ and so $$
  	|E(G)| = |E(G[A])| + |E(G[B])| -6 
  	\leq 4|A|  + 4|B| - 26 
  = 4|V(G)| - 10,
$$
a contradiction, finishing the argument in this case.

\vskip 5pt \noindent {\bf Case 1.2: $k=5$ and either  $G_A \simeq K_{2,2,2,2}$ or $G_B \simeq K_{2,2,2,2}$.}
 Assume without loss of generality that    $G_A \simeq K_{2,2,2,2}$. Let $A \cap B = \{v_1,v_2,v_3,v_4\}$ where $v_1v_2 \not \in E(G)$ and all the other pairs of vertices in $A \cap B$ are adjacent. Let $A - B = \{u_1,u_2,u_3,u_4\}$, where $v_i$ is the unique non-neighbor of $u_i$ in $G[A]$. Let $G^{+}_A$ be the graph obtained from $G$ by contracting any connected component of $B-A$ to a single vertex and deleting the remaining vertices of $B-A$, as we did in the previous case. Then $G^{+}_A$ is obtained from $G[A]$ by adding a vertex adjacent to all the vertices in  $A \cap B$. Contracting the edges $u_3v_1$ and $u_4v_2$ in $G^{+}_A$ we obtain $K_7^{\vee}$ as a minor of $G$. This contradiction finishes this case.
 
 \vskip 5pt \noindent {\bf Case 1.3: $k=4$, $G_A \simeq K_{2,2,2,2}$ and $G_B \simeq K_{2,2,2,2}$.}
  As in the previous case let $A \cap B = \{v_1,v_2,v_3,v_4\}$, and now let $B-A = \{w_1,w_2,w_3,w_4\}$ where $v_i $ is the unique non-neighbor of $w_i$ in $G[B]$. Assume without loss of generality that $v_3v_4 \in E(G)$. Contracting the edges $v_1w_2,v_2w_1$ and $w_3w_4$, we replace $G[B]$ by a clique of size five, with four vertices corresponding to $A \cap B$. It now follows as in the previous case that $G$ has a $K_7^{\vee}$ minor, again a contradiction, finishing the proof of this case and the claim.
 \end{proof}
		
\begin{claim}\label{c:no4sepNotV} There exists no $4$-separation $(A,B)$ of $G$ such that $|A-B| \geq 2$ and $|B-A| \geq 2$. 
\end{claim}	

\begin{proof} Suppose such a separation $(A,B)$ exists. By  \cref{c:no4sep2K4} we may assume that   $G[B]$ has no $(A \cap B)$-rooted $K_4$-minor and choose such a separation with $|B|$ maximum. By \cref{l:K4} we have $|E(G[B])| \leq  3|B| - 7$.
	
	The graph $G[B]$ has an $(A \cap B)$-rooted $K_4^{-}$-minor by Lemma \ref{c:RST64}. Let $G'$ be a minor of $G$ with $V(G')=A$ obtained by contracting $B$ onto $A \cap B$ so that $|E(G'[A \cap B])| \geq 5$. 
	Suppose first that $G'$ is $4$-connected. Then $|E(G')| \leq  4|V(G')| - 8$ as $G'$ is not an enemy. It follows that 
\begin{align*}|E(G)| &= |E(G')|- |E(G'[A \cap B])| + |E(G[B])|\leq (4|A|- 8) - 5 + (3|B|-7)  \\ &= 4(|V(G)|+4)-20 - |B| \leq 4|V(G)|-10,\end{align*}
a contradiction.	
	
Thus there exists a non-trivial $3$-separation $(C,D)$ of $G'$. If $A \cap B \subseteq C$ then $(D, C \cup B)$ is a non-trivial $3$-separation of $G$. Thus there exists $v_1 \in (A \cap B)-C$ and symmetrically there exists $v_2 \in (A \cap B)-D$. Then $\{v_1,v_2\}$ is the unique pair of non-adjacent vertices in $A \cap B$. In particular, the choices of $v_1$ and $v_2$ are unique, and so $|C \cap B | = |D \cap B| =  3$.

Suppose first that $C - (D \cup B) \neq \emptyset$. Then $(C, D \cup B)$ is a non-trivial $4$-separation of $G$. By the choice of the separation $(A,B)$ we have that  $G[D \cup B]$ has an  $(C \cap (D \cup B))$-rooted $K_4$-minor. Thus $G[C]$ has no  $(C \cap (D \cup B))$-rooted $K_4$-minor  by \cref{c:no4sep2K4} and so $|E(G[C])| \leq 3 |C| - 7$ by \cref{l:K4}. 
Similarly, assuming  $D - (C \cup B) \neq \emptyset$ we have $|E(G[C])| \leq 3 |D| - 7$.

Under these assumptions 
we have \begin{align*}|E(G)| &\leq |E(G[B])|+|E(G[C])|+|E(G[D])| \\ &\leq 3|B|+3|C|+3|D|-21 = 3(|V(G)|+7)-21 \leq 4|V(G)|-9,\end{align*}
where the last inequality holds by \cref{c:vertexnumber}.
This contradicts the choice of $G$.

Thus we suppose without loss of generality that   $C - (D \cup B) = \emptyset$. Then  $D - (C \cup B) \neq \emptyset$ as $|B -A| \geq 2$, and we have 
 \begin{align*}|E(G)| &\leq  |E(G[B])|+|E(G[D])|+1 \leq 3|B|+3|D|-13 \\ & = 3(|V(G)|+3)-13 =  3|V(G)|-4 \leq 4|V(G)|-9,\end{align*}
again a contradiction.
\end{proof}

\begin{claim}\label{c:K2222}
For every $uv \in E(G)$  
we have	$G / uv \not \simeq K_{2,2,2,2}$ 
\end{claim}
\begin{proof} Suppose that 	$G / uv  \simeq K_{2,2,2,2}$. Then $|V(G)| =9$ and so $|E(G)| \geq 28$. 
As  $$24 = |E(G / uv)| = |E(G)| - 1 - |N(u) \cap N(v)|  \geq 27 - |N(u) \cap N(v)|,$$
it follows that   $|N(u) \cap N(v)| \geq 3$. Let $V(G)=\{u,v,u_1,v_1,u_2,v_2,u_3,v_3,w\}$, where $u_iv_i$ for $i=1,2,3$ are the only non-edges in $G \setminus u \setminus v$. By symmetry, we may assume that $\{u_1,u_2,u_3\}  \subseteq N(u) \cap N(v)$ or $\{u_1,v_1,u_2\} \subseteq N(u) \cap N(v)$. In both cases  we may further assume by symmetry that $uv_1,vv_2 \in E(G)$, or $N(v) = \{u,u_1,u_2,u_3\}$ and $N(u) = V(G) - w$. In the first case,  
$G /uu_1/vv_2$ has a $K^-_7$ subgraph, a contradiction. In the second case $G/v_1v_2/v_3w \simeq K^{\vee}_7$, again a contradiction.
\end{proof}	
	
\begin{claim}\label{c:degree4} Let $uv \in E(G)$ be such that $\deg(v) = 4$. Then there exists $w \in V(G)$ such that $\deg(w)=4$ and $vw, uw \in E(G)$. \end{claim}

\begin{proof} 
	Let $G'= G / uv$. As  $$|E(G')| \geq |E(G)| -4  \geq 4|V(G)|-12 =  4|V(G')|-8,$$  $G' \not \simeq K_{2,2,2,2}$ by \cref{c:K2222}, and $G'$ is not an enemy by the choice of $G$, we conclude that $G'$ is not $4$-connected. Let $(A',B')$ be a non-trivial $3$-separation of $G'$. It corresponds to a non-trivial $4$-separation $(A,B)$ of $G$ such that $u,v \in A \cap B$. By \cref{c:no4sepNotV} we may assume $A-B = \{w\}$. Then the only neighbors of $w$ lie in $A \cap B$. As $G$ is $4$-connected, it follows that every vertex of $A \cap B$ is adjacent to $w$ and the claim holds.
\end{proof}	 

\begin{claim}\label{c:5connected} $G$ is $5$-connected.\end{claim}

\begin{proof} By \cref{c:no4sepNotV} it suffices to show that $G$ has no vertices of degree four. Suppose for a contradiction that this is false. As $|E(G)| = 4|V(G)| - 8 > 2|V(G)|$, $G$ also has vertices of degree larger than four.
	
As $G$ is connected, it follows that there exists $uv \in E(G)$ be such that $\deg(v) = 4$, $\deg(u) > 4$. By \cref{c:degree4} there exists 	$w \in V(G)$ such that $\deg(w)=4$ and $vw, uw \in E(G)$. By \cref{c:degree4} applied to the edge $vw$ there exists $w' \in V(G)$ such that 
$\deg(w')=4$, $vw',ww'\in E(G)$. Thus $(N(v) \cup N(w), V(G) - u - w)$ is a $4$-separation of $G$, contradicting \cref{c:no4sepNotV}.
\end{proof}

\begin{claim}\label{c:triangles} $|N(u) \cap N(v)| \geq 4$ for every $uv \in E(G)$.
\end{claim}	

\begin{proof}
Let $G'= G / uv$. The graph $G'$ is $4$-connected by \cref{c:5connected}. We have \begin{equation}\label{e:eg'} |E(G')|  + 1 + |N(u) \cap N(v)| = |E(G)| \geq 4|V(G)|-8  = 4 |V(G')|-4, \end{equation} and $G' \not \simeq K_{2,2,2,2}$ by \cref{c:K2222}.
As $G'$ is not an enemy, it follows that $|E(G')| \leq 4|V(G')| - 9$, which by \eqref{e:eg'} implies the claim. 
\end{proof} 	

\begin{claim}\label{c:edgenumber} $|E(G)| = 4|V(G)|-8$.
 \end{claim}	
 \begin{proof}
 Suppose for a contradiction that $|E(G)| > 4 |V(G)|-8$ and let $e \in E(G)$ be arbitrary. The graph $G - e$ is $4$-connected by \cref{c:5connected},  $|E(G-e)| \geq  4 |V(G-e)|-8,$ and $G - e \not \simeq K_{2,2,2,2}$ by \cref{c:vertexnumber}. Thus $G-e$ is an enemy, contradicting the choice of $G$.
\end{proof} 	

Let $v$ be a vertex of $G$ with minimum degree.
 
\begin{claim}\label{c:mindegree}  $5 \leq \deg(v)  \leq 7$.  
\end{claim}	

\begin{proof}
We have	$\deg(v) \geq 5$ by \cref{c:5connected}, and $\deg(v) \leq 7$ by \cref{c:edgenumber} .
\end{proof}

 \begin{claim}\label{c:degree5}  $ \deg(v) \neq 5$.  
 \end{claim}	
 
\begin{proof}	
 Suppose that $\deg(v)=5$.
 By \cref{c:triangles},  $N[v]$ is a clique in $G$. Let $C$ be a component of $G \setminus N[v]$.  (It exists by \cref{c:vertexnumber}.)	 By \cref{c:5connected}, every vertex of $N(v)$ has a neighbor in $C$. By contracting $C$ to a single vertex and deleting the remaining components of  $G \setminus N[v]$ we obtain $K^-_7$ as a minor of $G$, a contradiction.
\end{proof}	

\begin{claim}\label{c:degree6}  $ \deg(v) = 6$.  
\end{claim}

\noindent {\it Proof.}	
 Suppose not. Then by Claims \ref{c:mindegree} and \ref{c:degree5} we have $\deg(v)=7$. By Claims  \ref{c:vertexnumber} and \ref{c:5connected} there exists a component $C$ of $G \setminus N[v]$ such that at least five vertices of $N(v)$ have a neighbor in $C$.
 
 Let $H$ be the complement of $G[N(v)]$. By \cref{c:triangles} every vertex of $H$ has degree at most two. Thus $H$ is isomorphic to a subgraph of one of the following graphs:
 $C_7$, $C_6 \sqcup K_1$, $C_5 \sqcup K_2$, $C_4 \sqcup C_3$ and $C_3 \sqcup C_3 \sqcup K_1$, where $\sqcup$ denotes disjoint union.
 
 We consider the corresponding cases, in each case obtaining a contradiction by showing that $K^{\vee}_7$ is a minor of $G$.
 \begin{itemize}
 	\item
 If $H$ is a subgraph of $C_7$ with vertex set $\{u_1,u_2,\ldots,u_7\}$ in order, we assume by symmetry that $u_1$ and $u_2$ have a neighbor in $C$. Contracting the edge $u_3u_7$ and contracting $C$ onto $u_2$ to add an edge $u_1u_2$ to $G[N(v)]$ we obtain a minor of $G$ inducing $K^{\vee}_7$ on $N[v]$.
 \item   If $H$ is a subgraph of $C_6 \sqcup K_1$ with vertex set $\{u_1,u_2,\ldots,u_6\}$ of $C_6$ in order, similarly to the previous case we can assume that $u_1$ and $u_2$ have a neighbor in $C$.  Contracting  $u_3u_6$ and   adding an edge $u_1u_2$ to $G[N[v]]$ using $C$, we again conclude that  $G$ has a $K^{\vee}_7$ minor.
 \item   If $H$ is a subgraph of $C_5 \sqcup K_2$ with vertex sets $\{u_1,u_2,\ldots,u_5\}$ and $\{w_1,w_2\}$, respectively. Again we can assume that the edge $u_1u_2$ can be added to $G[N(v)]$ by contracting $C$, and contracting $u_5w_1$ we obtain a $K^{\vee}_7$ minor.
 \item   If $H$ is a subgraph of $C_4 \sqcup C_3$ with vertex sets $\{u_1,u_2,\ldots,u_4\}$ and $\{w_1,w_2,w_3\}$, respectively.  If $C$ has two neighbors in $\{w_1,w_2,w_3\}$, say $w_2$ and $w_3$, we add an edge $w_2w_3$ using $C$ and contract $w_1u_1$ to obtain a    $K^{\vee}_7$ minor. Otherwise $\{u_1,u_2,\ldots,u_4\}$ all have a neighbor in $C$ and we contract $C$ onto $u_1$ and contract $u_3w_1$. 
 \item  Finally, if $H$ is a subgraph of $C_3 \sqcup  \sqcup C_3 \sqcup K_1$ with vertex sets $\{u_1,u_2,u_3\}$ and $\{w_1,w_2,w_3\}$ of two cycles. Without loss of generality, $u_1$ and $u_2$ have a neighbor in $C$ and we add an edge $u_1u_2$ to $G[N[v]] $ using $C$ and contract $u_3w_1$ to obtain the required minor in this last case. \qed
 \end{itemize}  
\vskip 5pt
	
The remainder of the proof is occupied by the last, most involved case when the minimum degree of $G$ is six.

Recall that $v$ is a degree six vertex in $G$. By \cref{c:triangles} every vertex in $N(v)$ has at most one non-neighbor in $N(v)$.
Let $N(v)= \{u_1,w_1,u_2,w_2,u_3,w_3\}$, such that $u_1w_1,u_2w_2,u_3w_3$ are the only possible non-edges in $G[N(v)]$.
 
 \begin{claim}\label{c:nopaths} There do not exist vertex disjoint paths $P$ and $Q$ in $G$ such that  \begin{itemize}
 		\item $P$ and $Q$ have no internal vertices in $N[v]$,
 		\item $P$ has ends $u_i$ and $w_i$ and $Q$ has ends $u_j$ and $w_j$ for some $\{i,j\} \subseteq \{1,2,3\}$.
 \end{itemize}  
 \end{claim}	
 \begin{proof}	
 	If such paths exist, contracting them to single edges and deleting the rest of $G \setminus N[v]$ we obtain $K^-_7$ minor of $G$.
\end{proof}	
	
 \begin{claim}\label{c:nonedges} $u_iw_i \not \in E(G)$ for $i=1,2,3$. 
\end{claim}	
\begin{proof}	
Suppose for a contradiction that $u_3w_3 \in E(G)$.  Consider a component $C$ of  $G \setminus N[v]$. By \cref{c:5connected} at least five vertices of $N(v)$ have a neighbor in $C$. Without loss of generality, we assume that $u_1,w_1$ are among these neighbors. Then a path between $u_1w_1$ in $C$ together with a single edge path $u_3v_3$ contradict \cref{c:nopaths}.
\end{proof}

\begin{claim}\label{c:nocommonnbrs} There exists $i \in \{1,2,3\}$ such that $u_i$ and $w_i$ have no common neighbors in $V(G)-N[v]$. 
\end{claim}	
\begin{proof}	
Suppose not:  For every $i \in \{1,2,3\}$ there exists $x_i \in V(G) - N[v]$ such that $u_ix_i,w_ix_i \in E(G)$ 	
If $x_1 \neq x_2$ then the paths $u_1x_1w_1$ and $u_2x_2w_2$ contradict \cref{c:nopaths}. Thus we assume $x_1=x_2=x_3$.

By \cref{c:vertexnumber} there exists a component $C$ of $G \setminus (N[v] \cup \{x_1\})$. Again, at least five vertices in $N[v] \cup \{x\}$ have a neighbor in $C$, and so by symmetry we assume that $u_3$ and $w_3$ are among these neighbors. Then a path $u_1x_1w_1$ and a path from
$u_3$ to $w_3$ through $C$ contradict \cref{c:nopaths}. 
\end{proof}

By \cref{c:nocommonnbrs} we assume without loss of generality that $u_1$ and $w_1$ have no common neighbors in $V(G)-N[v]$.

\begin{claim}\label{c:no5separation} There exists no non-trivial $5$-separation $(A,B)$ of $G$ such that $N[v] \subseteq A$.
\end{claim}	

\noindent {\it Proof.}
	Suppose such a separation exists. By $5$-connectivity of $G$, $G[A]$ contains a collection $\mc{P}$ of five pairwise disjoint paths, each with one end in $N(v)$  the other end in $A \cap B$ and internally disjoint from $N[v]$. By symmetry we may assume that $u_1,w_1,u_2,w_2$ and $w_3$ be the ends of these paths in $N(v)$ and let $s_1,t_1,s_2,t_2$ and $x$ be the ends of the respective paths in $A \cap B$. We
	apply \cref{l:RST24} to $G'=G[B] \setminus x$ and  a sequence $s_1,s_2,t_1,t_2$ and distinguish cases based on the outcome that holds, obtaining a contradiction in each case.
	\begin{itemize}
		\item  If \cref{l:RST24} (i) holds, $G[B]$ contains a pair of vertex disjoint paths $Q_1$ and $Q_2$ such that $Q_i$ has ends $s_i$ and $t_i$ for $i=1,2$. Extending $Q_1$ and $Q_2$ using paths in $\mc{P}$ that share ends with them, we obtain a pair of paths contradicting \cref{c:nopaths}. 
		\item  If \cref{l:RST24} (ii) holds there exists a separation $(A',B')$ of $G'$ of order at most three such that $(A \cap B) \setminus \{x\} \subseteq A'$ and $B' -A' \neq \emptyset$. Then $(A \cup A' \cup \{x\}, B' \cup \{x\})$ is a non-trivial separation of $G$ of order at most four, contradicting \cref{c:5connected}.
		\item Finally, if \cref{l:RST24} (iii) holds then $G'$ can be drawn in a disc with vertices $s_1,s_2,t_1,t_2$ on the boundary in this order. Consider arbitrary edge  $yz \in E(G')$ such that $y \in B'-A'$. Then $y$ and $z$ have at most two common neighbors in $G'$, as otherwise one of them would be separated from the boundary of the disc by $y,z$  and another common neighbor of $y$ and $z$, and  \cref{l:RST24} (ii) holds, which is an already eliminated possibility. It follows that $|N(y) \cap N(z)| \leq 3$ in $G$, as $x$ is the only vertex that can be a common neighbor of $y$ and $z$ in $G$, but not $G'$. This contradicts \cref{c:triangles}. 		 \qed
	\end{itemize}  		    
\vskip 5pt

For the final step of the proof 
we apply \cref{l:RST24} to the graph $G' = G \setminus \{v,u_1,w_1\}$ and the sequence $u_2,u_3,w_2,w_3$ and again consider the cases based on the outcome. As in the proof of \cref{c:no5separation},  \cref{l:RST24} (i) contradicts \cref{c:nopaths}.

Suppose that  \cref{l:RST24} (ii) holds, i.e. there exists a separation $(A,B)$ of $G'$ of order at most three such that $|B \setminus A| \geq 2$ and $u_2,u_3,w_2,w_3 \in A$. Then $(A \cup\{v,u_1,w_1\}, B \cup \{u_1,w_1\})$ is a separation of $G$ contradicting \cref{c:no5separation}. Finally, if   \cref{l:RST24} (iii) holds then $G'$ can be drawn in a disc with vertices $u_2,u_3,w_2,w_3$ on the boundary in this order. In particular $|E(G')| \leq 3|V(G')|-7$. By our assumption, every vertex in $V(G) \setminus N[v]$ has at most one neighbor in  $\{v,u_1,w_1\}$, while $u_2,u_3,w_2$ and $w_3$ each have three neighbors in $\{v,u_1,w_1\}$. Finally, note that $G$ induces two edges on the set  $\{v,u_1,w_1\}$. Therefore
$$ |E(G)| \leq |E(G')| + (|V(G')| -4)  + 4\cdot 3 + 2 \leq 4|V(G')| + 3 = 4|V(G)|-9,
$$
a contradiction, completing the proof of \cref{t:extremal}.


\section {Proof of \cref{t:main}.}\label{s:main} 

In this section we prove~\cref{t:main}. \cref{t:extremal} serves as our main tool. We will also use \cref{t:KM2019} and  the following additional results.

Let $k \geq 0$ be an integer.
Generalizing the concept mentioned in the introduction, we say that a graph $G$ is \emph{$k$-contraction-critical}  if $\chi(G)=k$, but $\chi(G') < k$ for every minor $G'$ of $G$ such that $G' \not\simeq G$. We will use the following properties of contraction-critical graphs.
The first two are classical and widely used results of Dirac and Mader, respectively, the third one is an ingredient from the  Kawarabayashi-Toft proof of $6$-colorability of $K_{4,4}$- and $K_7$-minor-free graphs, which we also utilize. 

\begin{thm}[Dirac \cite{Dirac1960}]\label{t:Dirac} Let $G$ be a $k$-contraction-critical graph. Then for each $v \in V(G)$ every independent set in $G[N[v]]$ has size at most $\deg(v) - k + 2$.
\end{thm}

\begin{thm}[Mader \cite{Mader1967}]\label{t:Mader} For all $k \geq 7$ every  $k$-contraction-critical graph is $7$-connected.
\end{thm}

\begin{lem}[Kawarabayashi, Toft~{\cite[Lemma 3 (i)]{KT2005}}]
\label{t:KT}
Let $G$ be a $7$-contraction-critical graph let $Z \subseteq V(G)$ with $|Z|=2$. Suppose that there exist three $5$-cliques $L_1,L_2$ and $L_3$  in $G$ such that $L_1 \cap L_2 = L_2 \cap L_3 = L_1 \cap L_3 = Z$. Then $G$ has a $K_7$ minor. 
\end{lem}

Finally we will use a theorem of Kawarabayashi, Luo, Niu and Zhang on complete minors in graphs with a triple of nearly disjoint cliques  which generalizes another ingredient from~\cite{KT2005}. 

\begin{thm}[Kawarabayashi, Luo, Niu and Zhang \cite{KLNZ2005}]
	\label{t:KLNZ}
	Let $G$ be a $(k + 2)$-connected graph where $k \geq 5$. If $G$ contains three $k$-cliques, say $L_1, L_2, L_3$, such that $|L_1 \cup L_2 \cup L_3| \geq 3k-3$, then $G$ contains a $K_{k+2}$ as a minor.
\end{thm}

With these in hand, we embark on the proof of \cref{t:main}.

As in the proof of \cref{t:extremal} we consider for a contradiction a minor-minimal counterexample to the theorem, i.e. a graph $G$ such that
\begin{itemize}
	\item $\chi(G) \geq 7$, 
	\item $G$ has no $K_7^{\vee}$-minor, and
	 \item $G$ is minor-minimal subject to these conditions.
\end{itemize}

 Note that this implies that $G$ is $7$-contraction-critical. As in the proof of \cref{t:extremal} we proceed to establish properties of $G$ in a series of claims.

\begin{claim}\label{c:7connected}  $G$ is $7$-connected. In particular,  $\deg(v) \geq 7$ for every $v \in V(G)$. 
\end{claim}	

\begin{proof}
	Follows from~\cref{t:Mader}.
\end{proof}

\begin{claim}\label{c:K6vee}  $G$ does not contain a $K_6^{\vee}$ subgraph.	\end{claim}	

\begin{proof}
	Suppose $H$ is a subgraph of $G$ isomorphic to $K_6^{\vee}$. Let $C$ be a connected component of $G \setminus V(H)$. By \cref{c:7connected}  every vertex of $H$ has a neighbor in $C$. Contracting $C$ to a single vertex we obtain $K_7^{\vee}$ as a minor of $G$, a contradiction.
\end{proof}

\begin{claim}\label{c:degree7vtx}  $G$ has at least $18$ degree seven vertices.
\end{claim}	

\begin{proof}
	As $G$ is $4$-connected, has no  $K_7^{\vee}$ minor and $G \not \simeq K_{2,2,2,2}$, we have $|E(G)| \leq 4 |V(G)|-9$ by \cref{t:extremal}. Thus $\sum_{v \in V(G)} (\deg(v)-8) \leq -18$. By \cref{c:7connected} $\deg(v) \geq 7$  for every $v \in V(G)$, so $G$ must have at least $18$ degree seven vertices contributing negatively to this sum.
\end{proof}

Let $\mc{C}$ be the set of all cliques of size five in $G$.

\begin{claim}\label{c:nbrhood}  If $v \in V(G)$ is such that $\deg(v)=7$ then there exists $L \in \mc{C}$ such that $\{v\} \subseteq L \subseteq N[v]$. 
\end{claim}	

\begin{proof} 
	
Let $H=G[N(v)]$. Then $H$ has no independent set of size three by~\cref{t:Dirac} applied with $k=7$. If $H$ has a clique of size four then the claim holds, so suppose for a contradiction that it does not. Then, as shown in  \cite[Section 2]{KT2005}, and can be verified using moderately routine case analysis, $H$ must have a subgraph isomorphic to the Moser spindle, see \cref{fig:inflationsC5}. Let the vertices of this subgraph be labeled as in  \cref{fig:inflationsC5}. By symmetry, as $\{u_3',u_3,u_4,u_4'\}$ is not a clique in $H$, we assume that $u_3'u_4' \not \in E(G)$.

Let $G' = G/vu_4'/vu_3'$ and let $v'$ be the new vertex of $G'$ obtained by the contractions. By the choice of $G$ there exists a $6$-coloring $c : V(G') \to \{1,\ldots, 6\}$ of $G'$. Moreover, every such coloring uses all six colors on the set $\{v',u_1,\ldots,u_5\}$. Indeed, if one of the colors is not present on this set then we can extend $c$ to a coloring of $G$ by defining $c(v)$ to be this color  and $c(u_4')=c(u_3')=c(v')$. 

Assume for convenience that $c(u_i)=i$ for $1 \leq i \leq 5$.
We now apply \cref{t:KM2019} to $c$ restricted to the subgraph of $G'$ induced by $\cup_{i=1}^5 c^{-1}(u_i) \subseteq V(G) - \{v,u_3',u_4'\}$. We need to show that there exists a Kempe chain of $c$ containing  $u_i$ and $u_{i+1}$  for all $1 \leq i \leq 5$, where  $u_6=u_1$, by convention. Suppose not. Without loss of generality we assume that $u_1$ and $u_2$ belong to different components of the subgraph of $G$ induced by $c^{-1}(1) \cup c^{-1}(2)$. Switching the colors $1$ and $2$  on the component containing $u_1$, we obtain a coloring $c'$ of $G'$ such that $c'(u_1)=c'(u_2)$, contradicting the observation in the previous paragraph.

By \cref{t:KM2019} we conclude that $G \setminus \{v,u_3',u_4'\}$ has a  $(u_1,\ldots,u_5)$-rooted $C_5$-model.  Contracting the bags of this model, we obtain a minor of $G$ which induces $K^{\vee}_7$ on $N[v] - \{u_4'\}$, a contradiction, establishing the claim.
\end{proof}
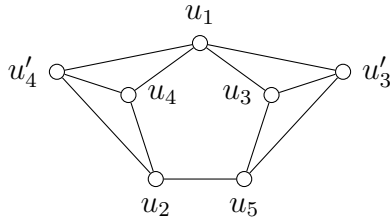
\begin{figure}[h]
	\centering
    \begin{tikzpicture}[roundnode/.style={circle, draw=black, inner sep=0pt, minimum size=2mm}]
	\node[roundnode,label=right:$u_4$] (1) at (162:1) {};
	\node[roundnode,label=above:$u_1$] (2) at (90:1) {};
	\node[roundnode,label=left:$u_3$] (3) at (18:1) {};
	\node[roundnode,label=below:$u_5$] (4) at (306:1) {};
	\node[roundnode,label=below:$u_2$] (5) at (234:1) {};
	\node[roundnode,label=left:$u_4'$] (6) at (162:2) {};
	\node[roundnode,label=right:$u_3'$] (7) at (18:2) {};
	
	\path
	(1) edge (2)
	(1) edge (6)
	(2) edge (3)
	(2) edge (6)
	(2) edge (7)
	(3) edge (4)
	(3) edge (7)
	(4) edge (5)
	(4) edge (7)
	(5) edge (1)
	(5) edge (6);
\end{tikzpicture}
\caption{ The graph considered in the proof of \cref{c:nbrhood} (The Moser spindle).}
\label{fig:inflationsC5}
\end{figure}

\cref{c:degree7vtx} and \cref{c:nbrhood} immediately imply the following.

\begin{claim}\label{c:cliqueunion}  $|\cup_{L \in \mc{C}} L| \geq 18.$ 
	\end{claim}

\begin{claim}\label{c:clique1}  $|L_1 \cap L_2| \leq 2$ for all pairs of distinct $L_1,L_2 \in \mc{C}$. 
\end{claim}	
\begin{proof} 
If $|L_1 \cap L_2|=4$ then $L_1 \cup L_2$ induces a $K_6^-$ subgraph in $G$, contradicting  \cref{c:K6vee}. 

Suppose now $|L_1 \cap L_2|=3$. As $G \setminus (L_1 \cap L_2)$ is $4$-connected by \cref{c:7connected}, there exist two vertex disjoint paths $P_1,P_2$ in  $G \setminus (L_1 \cap L_2)$  each with one end in $L_1 - L_2$ and the other in $L_2 - L_1$. Let $x_i$ and $y_i$ denote the ends of $P_i$ in  $L_1 - L_2$  and  $L_2 - L_1$, respectively. As $G \setminus (L_1 \cap L_2) \setminus \{x_1,y_2\}$ is $2$-connected there exists a path $Q$ in this graph with one end in $V(P_1)$ another in $V(P_2)$ and, otherwise, disjoint from $P_1 \cup P_2$. Contracting each of  $P_1 \setminus \{x_1\}$,  $P_2 \setminus \{y_2\}$ to as single vertex and contracting all but one edges of $Q$, we obtain a minor of $G$ inducing $K_7^-$ on $L_1 \cup L_2$, a contradiction, implying the claim.
\end{proof}

As $G$ has no $K_7$ minor, \cref{t:KLNZ} implies the following.

\begin{claim}\label{c:klnz}  $|L_1 \cup L_2 \cup L_3| \leq 11$ for all $L_1,L_2,L_3 \in \mc{C}$.
\end{claim}	

\begin{claim}\label{c:clique2}  $|L_1 \cap L_2| \geq 1$ for all $L_1,L_2 \in \mc{C}$. 
\end{claim}	
\begin{proof}
Suppose for a contradiction that there exist $L_1,L_2 \in \mc{C}$ such that $L_1 \cap L_2 = \emptyset$. By \cref{c:cliqueunion} there exists $L_3 \in \mc{C}$ such that $L_3 \not \subseteq L_1 \cup L_2$. By \cref{c:klnz} we have $|L_3 - (L_1 \cup L_2)| =1$, and by \cref{c:clique1} it follows that $|L_3 \cap L_2|=|L_3 \cap L_1|=2$. As $G \setminus L_3$ is $2$-connected there exist two vertex disjoint paths $P_1,P_2$ in  $G \setminus L_3$ each with one end in $L_1 - L_3$ and the other in $L_2 - L_3$. Contracting the edges of these paths we obtain a minor of $G$ with a $K_7^{\vee}$ subgraph induced by $L_3$ and the two vertices obtained by contraction, yielding the desired contradiction. 	 
\end{proof}

\begin{claim}\label{c:clique3} If $|L_1 \cap L_2| = 2$ for some $L_1,L_2 \in \mc{C}$ then $|L_1 \cap L_2 \cap L_3| \neq \emptyset$ for every $L_3 \in \mc{C}$.
\end{claim}	
\begin{proof}
	Suppose for a contradiction that there exist $L_1,L_2,L_3 \in \mc{C}$ such that $|L_1 \cap L_2|=2$ and  $|L_1 \cap L_2 \cap L_3| = \emptyset$. By \cref{c:clique2} we have $|L_1 \cap L_3|,|L_2 \cap L_3| \geq 1$. Let $v \in L_2 \cap L_3$ be arbitrary. Then $v$ is adjacent to the vertices in $L_2 \cap L_1$ and $L_1 \cap L_3$. It follows that $G[L_1 \cup \{v\}]$ has a $K_6^{\vee}$ subgraph, contradicting \cref{c:K6vee}. 
\end{proof}	

\begin{claim}\label{c:clique4} Let $L_1,L_2,L_3 \in \mc{C}$ be pairwise distinct then at least two of their pairwise intersections have size two and the remaining one has size one.
\end{claim}

\begin{proof} By Claims \ref{c:clique1} and \ref{c:clique2} we have $1 \leq |L_i \cap L_j| \leq 2$ for all $\{i,j\} \subseteq \{1,2,3\}$. If $|L_1 \cap L_2|=|L_2 \cap L_3| = |L_1 \cap L_3|=1$ then $|L_1 \cup L_2 \cup L_3| \geq 12$, contradicting \cref{c:klnz}. Thus we assume $|L_1 \cap L_2|=2$. If $|L_2 \cap L_3|=|L_1 \cap L_3| = 1$ then we still have $|L_1 \cup L_2 \cup L_3| \geq 12$ as $|L_1 \cap L_2 \cap L_3| \neq \emptyset$ by \cref{c:clique3}.

It remains to show exclude the case $|L_1 \cap L_2|=|L_2 \cap L_3|=|L_1 \cap L_3| = 2$. Let $Z = L_1 \cap L_2 \cap L_3$. If $|Z|=1$ then $L_1 \cup (L_2 \cap L_3)$ induces a supergraph of $K_7^{\vee}$, contradicting \cref{c:K6vee}. Finally, if $|Z|=2$ then $G$ has a $K_7$ minor by Lemma~\ref{t:KT}, again a contradiction.
\end{proof}	

We are now ready to finish the proof of \cref{t:main}. By Claims \ref{c:cliqueunion} and \ref{c:klnz} we have $|\mc{C}| \geq 5$. Let $H$ be a graph with $V(H) = \mc{C}$ and $L_1L_2 \in E(H)$ for $L_1,L_2 \in \mc{C}$ if and only if $|L_1 \cap L_2 |=2$. By \cref{c:clique4} the subgraph of $H$ induced by any three vertices has exactly two edges. Thus $H$ must be bipartite, as every induced odd cycle contains a triple of vertices inducing one or three edges, and $H$ can not contain independent set of size three. It follows that $4  \geq |V(H)|=|\mc{C}|$, a contradiction, finishing the proof of \cref{t:main}. 

\section{Concluding remarks}\label{s:remarks}

\noindent {\bf W}e have shown that every graph with no 	$K_7^{\vee}$-minor is $6$-colorable. It is natural to wonder whether similar methods could be used to prove  the symmetrical statement replacing $K_7^{\vee}$ by  $K_7^{=}$, the other graph obtained from $K_7$ by deleting two edges.

\begin{conj}\label{c:K7=}
	Every graph with no  $K_7^{=}$-minor is $6$-colorable.
\end{conj}

The main difficulty in proving Conjecture \ref{c:K7=} lies in proving an analogue of \cref{t:extremal}. The direct analogue does not hold, as the following example shows. Let $G_n$ be a graph with $V(G_n)=A \cup B$, where $|A|=4, |B|=n-4, A\cap B = \emptyset$, every vertex of $A$ is adjacent to every other vertex of $G_n$ and $E(G_n[B])$ is a matching of size $\lfloor (n-4)/2 \rfloor$. Then $G_n$ is $4$-connected, $|V(G_n)|=n$ and $|E(G_n)| = 4n + \lfloor n/2 \rfloor - 12$. Moreover for every minor $H$ of $G$ there exists $X \subseteq V(H)$ with $|X| \leq 4$, such that $E(H \setminus X)$ is a matching. In particular, $K_7^{=}$ is not a minor of $G_n$. 

It is possible that a version of \cref{t:extremal} for $K_7^{=}$-minors holds if we require higher connectivity.

\begin{conj}\label{c:K7=extremal}
	Every $5$-connected graph $G$ with $|E(G)| \geq 4|V(G)| - 9$ has a $K_7^{=}$-minor, unless $G \simeq K_{6}$.
\end{conj}

The bound in \cref{c:K7=extremal} would be optimal as every graph $G$ obtained from a $5$-connected planar triangulation by adding a universal vertex  satisfies $|E(G)| = 4|V(G)| - 10$ and has no $K_6$-minor, and thus no $K_7^{=}$-minor.

\vskip 10pt

\noindent {\bf T}he following is a natural common strengthening of \cref{t:main} and Conjecture \ref{c:K7=}. 

\begin{conj}\label{c:K7-}
	Every graph with no  $K_7^{-}$-minor is $6$-colorable.
\end{conj}

Again proving \cref{c:K7-} using  the strategy similar to the one used in this paper would require a corresponding extremal result. It is possible that a variant of \cref{c:K7=extremal} holds for $K_7^{-}$-minors with a longer list of small exceptional graphs.

\vskip 10pt
\noindent {\bf F}inally, let us mention existing relaxations of Hadwiger's conjecture for $t>6$ similar to \cref{t:Jakobsen} and \cref{t:main}, primarily due to Zi-Xia Song and co-authors.  

\begin{thm}[Lafferty and Song {\cite{LS2022}}] Every graph with no minor obtained from $K_8$ by deleting four edges is $7$-colorable.
\end{thm}	

\begin{thm}[Lafferty and Song {\cite{LS2022b}}] Every graph with no minor obtained from $K_9$ by deleting six edges is $8$-colorable.
\end{thm}

\begin{thm}[Rolek and Song {\cite{RS2017}}] Every graph with no  $K^{\vee}_8$-minor and no $K^{=}_8$-minor is $8$-colorable.
\end{thm}

\begin{thm}[Rolek and Song {\cite{RS2017}}] Every graph with no  $K^-_8$-minor is $9$-colorable.
\end{thm}

\begin{thm}[Rolek {\cite{Rolek2020}}] Every graph  with no no  $K^{\vee}_9$-minor and no $K^{=}_9$-minor is $10$-colorable.
\end{thm}

\bibliographystyle{alpha}
 \bibliography{references}

\end{document}